\documentclass[final,a4paper,12pt]{scrartcl}

\usepackage[british]{babel}

\usepackage{color}

\usepackage{amsmath}
\usepackage{amssymb}
\usepackage{enumerate}
\usepackage{amsthm}
\usepackage{bbm} 
\usepackage{soul}

\usepackage{mathrsfs} 

\makeatletter
\def\@swap#1#2{\let\@tempa#1\let#1#2\let#2\@tempa}
\@swap\epsilon\varepsilon
\makeatother

\newcommand{\RR}{\mathbb{R}}
\newcommand{\RRa}{\mathbb{R}^\ast}

\newcommand{\GG}{\mathscr{G}}

\newcommand{\as}{a.s.\thinspace}

\newcommand{\abs}[1]{\left\lvert #1 \right\rvert}

\DeclareMathOperator{\rRe}{Re}

\newcommand{\iu}{\mathrm{i}} 

\newenvironment{eqnarr*}{\begin{IEEEeqnarray*}{rCl}}{\end{IEEEeqnarray*}\ignorespacesafterend}

\usepackage{mathabx}

\usepackage[final,bookmarksopen=true]{hyperref}
\renewcommand{\eqref}[1]{\hyperref[#1]{(\ref*{#1})}}
\numberwithin{equation}{section}

\newcommand{\GGt}{(\GG_t)_{t \ge 0}}

\newcommand{\stP}{\mathrm{P}}

\newcommand{\stPs}{\mathrm{P}^\sharp}
\newcommand{\stPhs}{\mathrm{\hat P}^\sharp}
\newcommand{\stPz}{\stP^0}
\newcommand{\stE}{\mathrm{E}}
\newcommand{\stEh}{\mathrm{\hat E}}

\newcommand{\stEhs}{\mathrm{\hat E}^\sharp}

\newcommand{\MAPP}{\mathbb{P}}
\newcommand{\MAPE}{\mathbb{E}}
\newcommand{\MAPPs}{\mathbb{P}^\sharp}

\newcommand{\MAPPh}{\widehat{\mathbb{P}}}

\newcommand{\MAPPhs}{\widehat{\mathbb{P}}^\sharp}
\newcommand{\MAPEhs}{\widehat{\mathbb{E}}^\sharp}

\DeclareMathOperator{\diag}{diag}
\newcommand{\Had}{\circ}
\DeclareMathOperator{\sgn}{sgn}

\newcommand{\dd}{\mathrm{d}}

\newcommand{\rhoh}{\hat{\rho}}

\newcommand{\for}{\qquad}

\newcommand{\define}{\emph}

\newcommand{\Mn}{M}
\newcommand*{\MAPPc}[1]{\MAPP^{(#1)}}

\renewcommand{\abs}[1]{\lvert #1 \rvert} 

    \def\beq{\begin{eqnarray}}
    \def\eeq{\end{eqnarray}}
    \def\beqq{\begin{eqnarray*}}
    \def\eeqq{\end{eqnarray*}}

    \def\stP{\textbf{P}}
    \def\ud{\mathrm{d}}

\newcommand*{\pref}[1]{\hyperref[#1]{(\ref*{#1})}}
\newcommand*{\refpref}[2]{\hyperref[#2]{\ref*{#1}(\ref*{#2})}}

\theoremstyle{plain}
  \newtheorem{dummy}{***}[section]
  \newtheorem{thm}[dummy]{Theorem}

  \newtheorem{lem}[dummy]{Lemma}
\theoremstyle{definition}

  \newtheorem{assn}[dummy]{Assumption}
\theoremstyle{remark}

\numberwithin{equation}{section}
  
\usepackage[numbers]{natbib}
\usepackage{doi}

\newcommand{\email}[1]{\href{mailto:#1}{\nolinkurl{#1}}}

\title{Recurrent extensions of real-valued self-similar Markov processes.}
\author{\large%
H. Pant\'i%
  \footnote{
  \textsc{ Facultad de Matem\'aticas, Universidad Aut\'onoma de Yucat\'an. Anillo PerifŽrico Norte, Tablaje Cat. 13615, Colonia Chuburn‡ Hidalgo Inn, MŽrida Yucat‡\'an.}
  E-mail: \email{henry.panti@correo.uady.mx}}
\ \   J.\ C.\ Pardo%
  \footnote{
  \textsc{ Centro de Investigaci\'on en Matem\'aticas. Calle Jalisco s/n. C.P. 36240, Guanajuato, Gto., Mexico.}
  E-mail: \email{jcpardo@cimat.mx}}
\ \   V.\ M.\ Rivero%
  \footnote{
  \textsc{  Centro de Investigaci\'on en Matem\'aticas. Calle Jalisco s/n. C.P. 36240, Guanajuato, Gto., Mexico.}
  E-mail: \email{rivero@cimat.mx}}
}
\date{\today}

\begin{document}

\maketitle

\begin{abstract}
\textbf{ Abstract.} Let $X=(X_t, t\ge 0)$ be a self-similar Markov process taking values in $\mathbb{R}$ such that the state 0 is a trap.
 In this paper, we present a necessary and sufficient condition for the existence of a self-similar recurrent extension of $X$ that leaves 0 continuously. The condition is expressed in terms of the associated Markov additive process via the Lamperti-Kiu representation. Our results extend those of   Fitzsimmons \cite{Fitzsimmons2006} and Rivero (\cite{Rivero2005}, \cite{Rivero2007})  where  the existence and uniqueness of a recurrent extension for positive self similar Markov processes were treated. In particular, we describe the recurrent extension of a stable L\'evy process which to the best of  our knowledge  has not been studied before.

{\small

}
\end{abstract}

\section{Introduction and main results}
In his seminal work \cite{Lamperti1972}, Lamperti studied the structure of positive self-similar Markov processes (pssMp) and  posed the problem of determining those  pssMp that agree  with a given pssMp up to the time the latter process first  hits 0. Lamperti \cite{Lamperti1972}  answered this question in the special case of  Brownian motion killed at 0 and  he found that the class of those extensions which are self-similar consists of the reflecting and absorbing Brownian motions and the extensions which immediately after reaching 0 jump according to the measure $\mathrm{d} x/x^{\beta+1}$, $\beta\in (0,1)$. Voulle-Apiala \cite{Vuolle1994} used It\^o's excursion theory to study the general case and provided a  sufficient condition on the resolvent of pssMp for the existence of recurrent extensions that leave $0$ continuously.  The main contribution of Voulle-Apiala  to this problem consist on the existence of a unique entrance law under which there exists a unique recurrent self-similar Markov process which turns out to be an extension of the pssMp after it reaches 0. Motivated by Voulle-Apiala's result, Rivero \cite{Rivero2005} provided  a simpler sufficient condition for the existence of such recurrent extension and a more explicit description of the entrance law. The  sufficient condition found by  Rivero was determined in terms of the underlying L\'evy process in the so-called Lamperti's transform of pssMp. Motivated by the aforementioned studies, Fitzsimmons  \cite{Fitzsimmons2006} and Rivero \cite{Rivero2007} provided, independently, a necessary and sufficient condition  for the existence of recurrent extensions that leave $0$ continuously. To be more precisely, their main result can be stated as follows: \textit{A pssMp that hits 0 in a finite time admits a self-similar recurrent extension that leaves 0 continuously if and only if the L\'evy process in the Lamperti transformation satisfies the so-called Cram\'er's condition.}

Recently, Chaumont et al. \cite{CPR2013}  studied the structure of real valued self-similar Markov processes and  established a  Lamperti type representation for such class of processes up to their first hitting time of $0$ in terms of  Markov additive processes (MAP).  Hence it is natural to pose the same question of Lamperti for such class of processes, in other words our aim is  to determine those  real valued self-similar Markov processes that agree  with a given real valued self-similar Markov process (rssMp) up to the time the latter process first  hits 0.   In particular, we would like to describe   the recurrent extension of a stable L\'evy process with scaling parameter $\alpha\in (1,2)$ up to its  first hitting time of $0$, which to the best of  our knowledge  has not been studied before.

Our arguments follow a similar strategy as in  \cite{Rivero2007}, nevertheless the construction of the recurrent extension of a real valued self-similar Markov processes is not straightforward and requires a careful analysis. Indeed, some fluctuation properties of MAPs and real valued self-similar Markov processes are required in order to guarantee the existence  of the excursion measure as well as its characterization. For instance, a complete understanding of eigenfunctions of a MAP, and of the moments of  exponential functionals of such processes is needed in terms of their characteristics, since they  are strongly related with the description of the entrance law of the recurrent extension. We conjecture this strategy can also be applied for the $d$-dimensional case ($d\ge 2$),  where a Lamperti type representation has been obtained recently by Alili et al. \cite{ACGZ2016}, but it seems that a much deeper analysis and a good understanding of processes behind the Lamperti transform is required, as well as the description of its entrance law,  which according to Kyprianou, Rivero, Sengul and Yang \cite{KypYang}  is complicated due to the fact that the driving part in a MAP can be essentially any Markov process taking values in the $d$-dimensional sphere, and hence many technical assumptions would be needed.
 
To state our results precisely, we introduce some notation and recall some of the basic theory of real self-similar Markov processes. Let $\mathbb{D}$ be the space of c\`adl\`ag paths defined on $[0,\infty)$ with values in $\RR$, endowed with the Skorohod topology and the corresponding Borel $\sigma$-field $\mathcal{D}$. A family of distributions $(\mathbf{P}_x,  x\in \mathbb{R})$ on $(\mathbb{D}, \mathcal{D})$ is called strong Markov family on $\mathbb{R}$ if the canonical process $(X_t, t\ge 0)$ is a standard Markov process (in the sense of Blumenthal and Getoor \cite{BG1968}) with respect to  $(\mathcal{F}_t)_{t\ge 0}$,  the canonical right continuous filtration.  If additionally the  process satisfies the so-called
\define{scaling property}: for all $c>0$,  
\begin{equation} \label{eq: self-similarity}
\{(c X_{t c^{-\alpha}}, t \ge 0), \stP_x \} \overset{\text{Law}}{=} \{( X_{t }, t \ge 0), \stP_{cx} \}, \qquad\textrm{for}\quad \, x\in \RR,
\end{equation}
then, the process is called  \define{real self-similar Markov process} (rssMp). We denote by $T_0$, the first hitting time of  0 for the process $X$, i.e. 
\[
T_0=\inf\{t>0:X_t=0\},
\] and we will  assume $T_0<\infty$,  $\stP_x$-a.s. and then the process dies  i.e.  $0$ is a \define{cemetery point},  for all $x\in\RR$.

A crucial point in our arguments is the following time change representation of rssMp, due to Chaumont et al. \cite{CPR2013}, in terms of  a Markov additive process taking values in $\RR \times \{ -1, 1\}$, here denoted by $(\xi, J)$.  For simplicity, we write $\{\pm 1\}:= \{ -1, 1\}$ and set $\mathbb{R}^*:=\mathbb{R}\setminus\{0\}$. The so-called Lamperti-Kiu representation can be stated as follows:  let $x\in \mathbb{R}^*$ then, under $\stP_x$, the  rssMp 
$X$ can be represented as follows
\[
X_t\mathbf{1}_{\{t<T_0\}} = x \exp\Big\{\xi(\tau(|x|^{-\alpha}t))\Big\}J({\tau(|x|^{-\alpha}t)}),\qquad  t\ge 0,
\]
where
\[
\tau(t)=\inf\left\{s\ge 0: \int_0^s \exp\Big\{\alpha \xi(u)\Big\} \mathrm{d} u\ge t\right\}.
\]
Let  $\GGt$ be a standard
filtration. Recall that a \define{Markov additive process}, $(\xi,J),$ tacking values in $\RR  \times \{\pm 1\},$ with respect to $\GGt,$,  if it is is a  c\`adl\`ag process, $(J(t), t \ge 0)$ is a  two states continuous-time Markov chain  and
the following property is satisfied:
\textit{ for any $i \in \{\pm 1\}$, $s,t \ge 0$: 
given $\{J(t) = i\}$,
the pair $(\xi(t+s)-\xi(t), J(t+s))$ is independent of 
$\GG_t$ and has the same distribution as $(\xi(s)-\xi(0), J(s))$
given $\{J(0) = i\}$.} 

If the MAP is killed, then $\xi$ shall be set to $-\infty$. We let $\MAPP_{z,i}$ be the law of $(\xi,J)$ started from the state $(z,i)$, and if $\mu$ is a probability distribution on $\{\pm 1\}$, we write
\[
\MAPP_{z,\mu} (\cdot)=  \sum_{i \in \{\pm 1\}} \mu(i) \MAPP_{z,i}(\cdot).
  \] 
We adopt a similar convention for expectations. It is well-known that a MAP $(\xi,J)$  can also be described in the following way (see for instance Asmussen \cite[\S XI.2a]{Asmussen2003} and Ivanovs \cite[Proposition 2.5]{Iva-thesis}):
 for $i,j\in \{\pm 1\}$, 
  there exists a sequence of iid L\'evy processes
  $(\xi_i^n)_{n \ge 0}$ and  a sequence of iid random variables
  $(U_{ij}^n)_{n\ge 0}$, independent
  of the chain $J$, such that if $S_0 = 0$
  and $(S_n)_{n \ge 1}$ are the
  jump times of $J$, the process $\xi$ has the representation
  \[ \xi(t) =\begin{cases}
  \xi(S_n -) + U_{J(S_n-), J(S_n)}^n + \xi_{J(S_n)}^n(t-S_n), & \text{if }\quad  t\in[S_n,S_{n+1}), t<\mathbf{k},\\
  -\infty, &\text{if }\quad t\ge \mathbf{k},
  \end{cases}
 \]
where the killing time $\mathbf{k}$ is the first time one of the appearing L\'evy processes is killed. Roughly speaking the behaviour of a MAP can be described as follows: if $J$ is in state $1$, then $\xi$ evolves according to a copy of $\xi_1$, a  L\'evy process. Once $J$ changes from $1$ to $-1$, which happens at rate $q_{1,-1}$, $\xi$ has an additional transitional jump and until the next jump of $J$, $\xi$ evolves according to a copy of $\xi_{-1}$. The MAP is killed as soon as one of the L\'evy processes is killed. 
Consequently, the mechanism behind the Lamperti-Kiu representation is simple: the Markov chain $J$
governs the sign
of the rssMp and on intervals with constant sign the Lamperti-Kiu representation simplifies to the Lamperti representation.

Hence in order to describe a MAP on $\RR \times \{\pm 1\}$, we require five characteristic components which are mutually independent: two  possibly killed L\'evy processes, say  $\xi_1=(\xi_1(t), t\ge 0)$ and $\xi_{-1}=(\xi_{-1}(t), t\ge0)$, two  random variables defined on $\mathbb{R}$ , say $U_{1, -1}$ and $U_{-1, 1}$ and a $2\times 2$ intensity matrix $Q = (q_{ij})_{i,j \in \{\pm 1\}}$, which is the  transition rate matrix of the chain $J$. 

Before we  introduce the matrix exponent of a MAP, we establish the convention that all matrices appearing in this work are written in the following form
   \[
   A = 
   \begin{pmatrix}
   a_{1 1} & a_{1 -1} \\
   a_{-1 1} & a_{-1 -1}
   \end{pmatrix}.
   \]  We also denote by $A^T$ for its transpose.
Let $\psi_{-1}$ and $\psi_{1}$ be the Laplace exponent of $\xi_{-1}$ and $\xi_{1}$, respectively (when they exist). For $z\in \mathbb{C}$,  let $G(z)$ denote the matrix whose entries are given  by $G_{ij}(z) = \MAPE[e^{z U_{ij}}]$ (when they exist) for $i\ne j$, and for $i=j$, $G_{ij}(z)=1$. For $z\in \mathbb{C}$, when it exist, we define 
 \begin{equation}\label{e:MAP F}
 F(z) = \diag(  \psi_{1}(z), \psi_{-1} (z))
  + Q \Had G(z),
\end{equation}
where $\Had$ indicates element-wise multiplication  also known as Hadamard multiplication. A straightforward computation yields for $t\ge 0$,   
\[ \MAPE_{0,i}[ e^{z \xi(t)} ; J(t)=j] = \bigl(e^{F(z) t}\bigr)_{ij} , \for i,\,j \in \{\pm 1\}, \] see Proposition 2.1, section XI.2 in \cite{Asmussen2003}.
For this reason, $F$ is called the \define{matrix exponent} of the MAP  $(\xi, J)$.

We will also be interested on the dual process of $(\xi, J)$, here denoted by  $((\xi, J), \MAPPh)$. Whilst the dual of a L\'evy process is equal in law to  its negative, the situation for MAPs is a little more involved.  The dual process is the MAP with probabilities $\widehat{\mathbb{P}}_{x, i}$, for $(x, i)\in \mathbb{R}\times \{\pm 1\}$, and whose Matrix exponent, whenever it is well-defined, is given by 
\[
\widehat{\mathbb{E}}_{0, i}\Big[e^{z\xi(t)}; J(t)=j \Big]=\left(e^{\widehat{F}(z)t}\right)_{ij}, \for i,\,j \in \{\pm 1\}, 
\]
where
\[
\widehat{F}(z)=\diag(  \psi_{1}(-z), \psi_{-1} (-z))
  + \widehat{Q} \Had G(-z)^T,
\]
and $\widehat{Q}$ is the intensity matrix of the modulating Markov chain on $\{\pm 1\}$ with entries given by 
\[
\widehat{q}_{ij }
=\frac{\pi_j}{\pi_i}q_{ji},\qquad  i,j \in \{\pm 1\},
\]
where $\pi=(\pi_1, \pi_{-1})$ is the invariant distribution associated to $J$. Note that the latter can also be written $\widehat{Q}=\Delta_\pi^{-1} Q^T \Delta_\pi$ where $\Delta_\pi = \diag(\pi_1, \pi_{-1})$, the matrix with diagonal entries given by $\pi$ and zeros everywhere else. Hence, when it exists, 
\begin{equation}\label{expduality}
\widehat{F}(z)=\Delta_\pi^{-1} F(-z)^T \Delta_\pi.
\end{equation}
Equivalently, we have
\[
\pi_i\widehat{\mathbb{E}}_{0,i}\Big[e^{\lambda \xi(t)}; J(t)=j\Big]=\pi_j{\mathbb{E}}_{0,j}\Big[e^{-\lambda \xi(t)}; J(t)=i\Big].
\]
According to Dereich et al. \cite{DK2015}, we have the following time reversal property between  $\xi$ and its dual, which will be relevant for our purposes. For any $t>0,$ fixed
\begin{equation}\label{timereversal}
\left(\{(\xi(t-s)-\xi(t), J((t-s)-)):s\le t\}, \mathbb{P}_{0,\pi}\right)\stackrel{\text{Law}}{=}\left(\{(\xi(s), J(s)):s\le t\}, \widehat{\mathbb{P}}_{0,\pi}\right).
\end{equation}

Another important property for our purposes is the construction of exponential martingales of MAP. It is known, see for instance \cite[\S XI.2c]{Asmussen2003} and \cite[Proposition 2.12]{Iva-thesis}, that the  matrix $F(z),$ for $z$ real, when it exists,  has a real simple eigenvalue $\kappa(z)$, which is smooth and convex on its domain and larger than the real part of all its other eigenvalues. Furthermore, the
corresponding right-eigenvector $v(z)$ may be chosen so that
$v_i(z) > 0$ for every $i \in E$,
and normalised such that
\begin{equation}
  \pi v(z) = 1. \label{e:h norm}
\end{equation}
The leading eigenvalue is sometimes also called the
\define{Perron-Frobenius eigenvalue} and it   identifies a martingale (also known as the Wald martingale) which allow us to define a change of measure
analogous to the  Esscher transform for L\'evy processes; cf.\ \cite[Proposition XI.2.4, Theorem XIII.8.1]{Asmussen2003}.
More precisely, if
\[ M(t,\gamma) = e^{\gamma \xi(t) - \kappa(\gamma)t}
    \frac{v_{J(t)}(\gamma)}{v_{J(0)}(\gamma)} ,
  \for t \ge 0, \]
for some $\gamma\in\RR$ such that the right-hand side is defined, then $M(\cdot,\gamma)$ is a unit-mean martingale with respect to $\GGt$
  under any initial distribution of $(\xi(0),J(0))$. Thus, we can define the change of measure
  \begin{equation}\label{expchange} \left.\frac{\dd \MAPPc{\gamma}}{\dd \MAPP}\right\rvert_{\GG_t}
    = \Mn(t,\gamma) . 
    \end{equation}
Moreover, under the probability measure  $\MAPPc{\gamma}$, the process $ (\xi, J)$ is a MAP with matrix exponent $F^{(\gamma)}$ and its leading eigenvalue is given by $\kappa^{(\gamma)}(z) = \kappa(z+\gamma)-\kappa(\gamma)$.

In the sequel, we consider the following assumption that  is also known as \define{Cram\'er's condition}.

\begin{assn}
\label{Cramer} 
 There exist a $\theta>0$ and a vector $(v_1, v_{-1})$ 
  \[
  \sum_{j\in  \{\pm 1\}} \MAPE_{0,i}\Big[e^{\theta \xi(t)}; J(t)=j\Big] v_j = v_i, \qquad i\in \{\pm 1\}.
  \]
The number $\theta$ is called the \emph{Cram\'er number}.
\end{assn}

 The latter conditon implies that $F(\theta)$ exists and $\kappa(\theta)=0$. Furthermore, applying  Jensen's inequality to $\MAPE_{0,i}[e^{z \xi(t)}; J(t)=j]$, for $z\in (0, \theta)$, allows us to deduce that $F(z)$ is well defined on $(0, \theta)$, and since $\kappa$ is a convex function on its domain, it follows $\kappa(z) \leq 0$ for $z\in (0,\theta]$. Conversely, if there exists a number $\theta>0$ such that $F(z)$ is well defined on $(0,\theta]$ with $\kappa(\theta)=0$, then Assumption \ref{Cramer} holds. In other words, another equivalent way to state Assumption \ref{Cramer} is as follows: there exists a number $\theta>0$ such that $F(z)$ is well defined on $(0, \theta]$ and $\kappa(\theta)=0$.

We set $\MAPPs:=\MAPPc{\theta}$, where $\MAPPc{\theta}$ is  defined by the exponential change of measure introduced in (\ref{expchange})  with $\theta$ satisfying  the Cram\'er condition. Recall that under $\mathbb{P}^\sharp$, the process $ (\xi,J)$ is a MAP with  $\kappa^\sharp(z) = \kappa(z+\theta)$ and denote for its dual by  $((\xi, J), \widehat{\mathbb{P}}^\sharp)$ which  is a MAP  whose leading eigenvalue  is  such that  $\widehat{\kappa}^\sharp(z) = \kappa^\sharp(-z) $. Furthermore, if $I$ denotes the exponential functional of  $\xi$, i.e., 
 $$I = \int_0^\infty \exp\{\alpha \xi(t)\} \mathrm{d} t,$$
then $\MAPEhs_{0,i}\left[ I^{\theta/\alpha - 1} \right]<\infty$, for $i\in \{\pm 1\}$, as it will be seen in Lemma~\ref{expecI}.

We now formally introduce the notion of a \define{recurrent extension}.  Let $(X, \mathbf{P})$ be a rssMp, defined as above,  and $T_0$ its first hitting time to $0$, we will refer to $(X, T_0)$ as the minimal process under $\mathbf{P}_x$. We say that a real valued  Markov process $(\overline{X}, \overline{\mathbf{P}})$ satisfying the scaling  property is a recurrent extension of $(X, \mathbf{P})$ provided that it behaves like the minimal process up to its first hitting time to 0 and  for which the state 0 is a regular and recurrent state.

We say that a $\sigma$-finite measure $\mathbf{n}$ on $(\mathbb{D}, \mathcal{F}_\infty)$ having infinite mass is an \define{excursion measure} compatible with $(X,\stP)$ if the following are satisfied:
\begin{enumerate}
\item $\mathbf{n}$ is carried by
  \[
  \{\omega \in \mathbb{D} : T_0(\omega) > 0, X_{t}(\omega)=0, \forall\, t\geq T_0\};
  \]
\item for every bounded $\mathcal{F}_\infty$-measurable $H$ and each $t>0$, $\Lambda\in \mathcal{F}_t$
  \[
  \mathbf{n}(H\circ \theta_t, \Lambda \cap \{ t< T_0 \}) = \mathbf{n}(\mathbf{E}_{X_t}(H), \Lambda \cap \{ t < T_0\} ),
  \]
where $\theta_t$ denotes the shift operator;
\item $\mathbf{n}(1-e^{-T_0}) < \infty$.
\end{enumerate}
 In the case that  the measure $\mathbf{n}$ only satisfies  properties 1 and 2, then it is called a \define{pseudo-excursion measure}. If the excursion measure $\mathbf{n}$ is such that $\mathbf{n}(1-e^{-T_0}) = 1$, then it is known as \define{normalized excursion measure}. Moreover, we say that $\mathbf{n}$ is \define{self-similar} if it has the following scaling property: there exists  $\gamma\in(0,1)$ such that for all $a>0$, it holds
  \[
  H_a\mathbf{n} = a^{\gamma \alpha} \mathbf{n},
  \]
where the measure $H_a\mathbf{n}$ is the image of $\mathbf{n}$ under the mapping $H_a:\mathbb{D} \to \mathbb{D}$, defined by $H_a(\omega)(t) = a\omega (a^{-\alpha} t)$, $t\geq 0$. The parameter $\gamma$ is called the \define{index of self-similarity} of $\mathbf{n}$.

We say that  the recurrent extension $(\overline{X}, \overline{\mathbf{P}})$ for which $0$ is a regular and recurrent state \define{leaves continuously} (resp.,  \define{by a jump}) the state $0$ whenever its excursion measure $\mathbf{n}$ is carried by the paths that leave $0$ continuously (resp., that leave 0 by a jump), i.e.
\[
\mathbf{n} (X_{0+}>0)=0 \qquad (\text{resp., }\mathbf{n} (X_{0+}=0)=0 ).
\]
 
We now state our main results.  Our first main results claims that the Cram\'er condition is necessary and sufficient  for the existence of a recurrent extension that leaves 0 continuously. 

\begin{thm} \label{thm: ext}
Let $(X, \mathbf{P})$ be a  rssMp with index $\alpha>0$. Suppose that $(X,\mathbf{P})$ hits its cemetery point 0 in a finite time $\mathbf{P}$-$\as$,  and let $((\xi, J), \MAPP)$ be the MAP  associated to $(X, \mathbf{P})$ via the Lamperti-Kiu representation. Then the following are equivalent:
\begin{enumerate}
\item[(i)] there exist a Cram\'er number $\theta\in (0,\alpha)$;

\item[(ii)] there exist a recurrent extension of $(X,\mathbf{P})$ that leaves 0 continuously and such that its associated excursion measure away from 0, say $\mathbf{n}$, is such that
  $$\mathbf{n}(1-e^{-T_0}) = 1.$$
\end{enumerate}
In this case, the recurrent extension in (ii) is unique, up to normalisation of the local time, and the entrance law associated with the excursion measure $\mathbf{n}$ satisfies, for any $f$ bounded and  measurable,
\begin{equation}
\label{entlaw}
\begin{split}
\mathbf{n}\left( f(X_t), t<T_0 \right) &= \frac{1}{C_{\alpha, \theta} t^{\theta/\alpha}} \left( v_1 \pi_1 \MAPEhs_{0,1}\left[ f\left( \frac{t^{1/\alpha}}{I^{1/\alpha}}\right) I^{\theta/\alpha - 1}  \right] \right. \\
&\hspace{4cm}\left. + v_{-1}\pi_{-1} \MAPEhs_{0,-1}\left[ f\left( - \frac{t^{1/\alpha}}{I^{1/\alpha}}\right) I^{\theta/\alpha - 1}  \right]\right), 
\end{split}
\end{equation}
where $\theta$ is the Cram\'er number  and
  \[
  C_{\alpha, \theta} = \Gamma(1-\theta/\alpha) \left( v_1 \pi_1 \MAPEhs_{0,1}\left[ I^{\theta/\alpha - 1} \right] + v_{-1}\pi_{-1} \MAPEhs_{0,-1}\left[ I^{\theta/\alpha - 1} \right] \right).  
  \]
\end{thm}

Our second main result provides   necessary and sufficient conditions  on the underlying MAP for the existence of recurrent extensions of rssMp that leave 0 by a jump.

\begin{thm}
\label{thm: jumpext}
For $\beta \in (0,\alpha)$, the following are equivalent:
\begin{enumerate}
\item[(i)]$\kappa$ is well defined in $\beta$ and $\kappa(\beta)<0$.

\item[(ii)] $\MAPE_{0,i}[I^{\beta/\alpha}] < \infty$, for $i\in \{\pm 1\}$.

\item[(iii)] The pseudo-excursion measure $\mathbf{n}^j = \mathbf{P}_\eta$, based on the jumping-in measure $$\eta(dx) = \abs x^{-(\beta+1)}\dd x, \qquad \textrm{for }x\neq 0,$$ is an excursion measure.

\item[(iv)] The minimal process $(X,T_0)$ admits a recurrent extension  that leaves 0 by  a jump and whose associated excursion measure $\mathbf{n}^\beta$ satisfies  $$\mathbf{n}^{\beta}\left(X_{0+}\in \dd x\right) = b_{\alpha, \beta}^{[x]} \abs x^{-(\beta+1)}\dd x,$$ where $[x]=\mathrm{sign}(x)$ and  $b_{\alpha, \beta}^1, b_{\alpha, \beta}^{-1}$ are such that 
  \[
  b_{\alpha, \beta}^1 \MAPE_{0,1}[I^{\beta/\alpha}] + b_{\alpha, \beta}^{-1} \MAPE_{0,-1}[I^{\beta/\alpha}] = \frac{\beta}{\Gamma(1-\beta/\alpha)}.
  \] 
\end{enumerate}
\end{thm}

 The remainder of this paper is organized as follows. In Section 2  three examples are presented where the main results are applied. In particular, we describe the recurrent extension that leaves 0 continuously (or by a jump) for the stable process and for  spectrally negative rssMp. Section 3 is devoted to establish and prove an existence theorem for recurrent extension of Markov processes. In the same section, some properties for self-similar excursion measures are given. Finally, Section 4 is devoted to the proofs of the main results.

\section{Examples} 
In this section, we apply our main results to  three particular classes of rssMp:  the so called stable processes, the MAP-dual stable processes (that we introduce below)  and  the spectrally negative case (or with no positive jumps).

\subsection{Stable processes.}
Let $(X,\mathbb{Q})$ be a stable process with scaling index $\alpha\in(0,2]$ starting from 0, i.e.   a L\'evy process which
also satisfies the \define{scaling property}. It is known that  the case $\alpha = 2$
corresponds to Brownian motion, which is excluded since its recurrent extensions have already been characterized by Lamperti \cite{Lamperti1972}. It is also known that 
the characteristic exponent  a stable process satisfies 
\begin{equation}\label{e:stable CE}
  \Psi(\theta) :=-\log \mathbb{Q} [e^{\iu\theta X_1}]=
  c\abs{\theta}^\alpha
  (1  - \iu\beta\tan\tfrac{\pi\alpha}{2}\sgn\theta), \for  \theta\in\RR,
\end{equation}
where $\beta = (c_+- c_-)/(c_+ + c_-)$,
$c = - (c_++c_-)\Gamma(-\alpha)\cos (\pi\alpha/2)$ and  $c_+,\, c_- $ are the two positive constants that appear on its associated L\'evy density. For more details of these facts, see  for instance Kuznetsov et al. \cite{KKPW2014} and Sato \cite[\S 14]{Sato}.

For consistency with the literature that we shall appeal to in this article, we shall always parametrize our $\alpha$-stable process in such a way that 
\[ c_+ = \Gamma(\alpha+1) \frac{\sin(\pi \alpha \rho)}{\pi} \quad \text{and} \quad
  c_- = \Gamma(\alpha+1) \frac{\sin(\pi \alpha \rhoh)}{\pi },
  \]
where
$\rho := \mathbb{Q}(X_t \ge 0)$ is the positivity parameter, 
and $\rhoh = 1-\rho$.  In that case, the constant $c$ simplifies to just $c = \cos (\pi\alpha(\rho - 1/2))$. 

It is well known that for $\alpha\in(0,1]$,  stable processes do not hit points and in particular they do not hit the point $0$. On the other hand,  for $\alpha\in(1,2)$,  stable processes make infinitely many jumps across a point, say $z$, before the first hitting time of $z$. Moreover, stable processes are transient for $\alpha\in(0, 1)$ and oscillate otherwise. Since we are interested in rssMp up to its first  hitting time of 0,  we will assume that  $\alpha \in (1,2)$. Nonetheless, it is important to point out that the computations below  holds for any value of $\alpha$.

 Let $(X, \stP)$ be the stable process killed up to its first hitting time of $0$ and we denote by  $((\xi, J), \MAPP)$ its  associated MAP  via the Lamperti-Kiu representation.
According to Kuznetsov et al. \cite{KKPW2014}, one can compute explicitly the matrix exponent of $((\xi, J), \MAPP)$ which satisfies $$
   F(z) =\left(
  \begin{matrix}
  - \displaystyle{ \frac{ \Gamma(\alpha - z) \Gamma(1 + z) }{ \Gamma(\alpha \hat{\rho} - z)\Gamma(1-\alpha \hat{\rho} + z) } } & \displaystyle{ \frac{ \Gamma(\alpha  - z) \Gamma(1 +z) }{ \Gamma(\alpha \hat{\rho})\Gamma(1-\alpha \hat{\rho})} } \\
   & \\
  \displaystyle{ \frac{ \Gamma(\alpha - z) \Gamma(1 + z) }{ \Gamma(\alpha \rho)\Gamma(1-\alpha \rho)} }  & -\displaystyle{ \frac{ \Gamma(\alpha - z) \Gamma(1 + z) }{ \Gamma(\alpha \rho -  z)\Gamma(1-\alpha \rho + z) } }
  \end{matrix}
  \right),
  $$
for $\rRe(z)\in (-1, \alpha)$. Using the reflection  identity twice, 
\[
\Gamma(1-z)\Gamma(z) = \frac{\pi}{\sin(\pi z)}, \qquad \textrm{for}\quad z\notin \mathbb{Z},
\]
and Ptolemy's identity, 
\[
\sin(\delta_1+\delta_2)\sin(\delta_2+\delta_3)=\sin(\delta_1)\sin(\delta_3)+\sin(\delta_1+\delta_2+\delta_3)\sin(\delta_2),
\]
with $\delta_1=\pi\alpha\rho$, $\delta_2=-\pi z$ and $\delta_3=\pi\alpha\hat\rho$; we conclude  that
\[
\det(F(z))=-\frac{ \Gamma(\alpha-z) \Gamma(1+z)}{ \Gamma(-z)\Gamma(1-\alpha+z)}\qquad \textrm{for} \quad z\in (-1, \alpha)\setminus \mathbb{Z}.
\]
The latter identity implies that for $\alpha\in (1,2)$, the Cram\'er number of  $((\xi, J), \MAPP)$ is $\theta=\alpha-1$. Hence, applying Theorem \ref{thm: ext}, the stable process with scaling index $\alpha\in(1,2)$ has a unique recurrent extension that leaves $0$ continuously.

\subsection{The MAP-dual of  a stable process and the stable process conditioned to be continuously absorbed at the origin.}\label{dualstable}
In this example, we  consider the case when stable processes are transient and do not hit points i.e. that $\alpha\in (0,1)$. In this  case, the process $(X, \stP)$ never hits $0$ and its radial part $|X|$ drifts to $+\infty$. In other words, its associated  MAP $((\xi, J), \MAPP)$,  via the Lamperti-Kiu representation, drifts to $+\infty$.
 In this particular example, we are interested in the  dual process $((\xi, J), \MAPPh)$ which in turn drifts to $-\infty$. We introduce its associated rssMp  $(X, \widehat{\mathbf{P}})$, via the Lamperti-Kiu representation, which we refer as the MAP-dual stable process. It is important to observe that the latter process reaches the point zero at finite time.

In order to compute the matrix  exponent $\widehat{F}$, we first observe that  using some explicit computations from \cite{CPR2013} and \cite{KP2013}, we can get explicitly the stationary distribution of $J$. More precisely,  we have
 \begin{align*}
\pi_1 &= k(\alpha) \Gamma( \alpha \hat{\rho} ) \Gamma( 1- \alpha \hat{\rho} ),  \quad 
\pi_{-1} = k(\alpha)\Gamma( \alpha \rho ) \Gamma( 1- \alpha \rho ),
\end{align*}
with  $k^{-1}(\alpha)=\Gamma( \alpha \rho ) \Gamma( 1- \alpha \rho ) + \Gamma( \alpha \hat{\rho} ) \Gamma( 1- \alpha \hat{\rho}) $. Thus,  from identity (\ref{expduality}) and straightforward computations, we deduce 
\begin{eqnarray*}
  \widehat{F}(z) =
\left(
  \begin{matrix}
  - \displaystyle{ \frac{ \Gamma(\alpha + z) \Gamma(1 - z) }{ \Gamma(\alpha \hat{\rho} + z)\Gamma(1-\alpha \hat{\rho} - z) } } & \displaystyle{ \frac{ \Gamma(\alpha + z) \Gamma(1 - z) }{ \Gamma(\alpha \hat{\rho})\Gamma(1-\alpha \hat{\rho})} } \\
   & \\
  \displaystyle{ \frac{ \Gamma(\alpha + z) \Gamma(1-z) }{ \Gamma(\alpha \rho)\Gamma(1-\alpha \rho)} }  & -\displaystyle{ \frac{ \Gamma(\alpha + z) \Gamma(1 - z) }{ \Gamma(\alpha \rho +  z)\Gamma(1-\alpha \rho - z) } }
  \end{matrix}
  \right),
\end{eqnarray*}
for $\rRe(z)\in (-\alpha,1)$. Furthermore, 
\[
\det(\widehat{F}(z)) = \det(F(-z))=\frac{\Gamma(\alpha+z) \Gamma(1-z)}{ \Gamma(z)\Gamma(1-\alpha-z)}\qquad \textrm{for} \quad z\in (- \alpha,1)\setminus \mathbb{Z},
\] implying that its Cram\'er number is $\theta=1-\alpha$. Hence, from Theorem \ref{thm: ext} we observe that  the MAP-dual  stable process $(X, \widehat{\mathbf{P}})$   has a unique recurrent extension that leaves $0$ continuously whenever $\alpha\in (1/2,1)$.

It is important to point out that there is a relationship between the former process and  the stable process conditioned to be continuously absorbed at the origin, for $\alpha\in(0,1)$. More precisely, according to Kyprianou et al.
 \cite{KRS2015},   the law of the stable process conditioned to be  continuously absorbed at the origin, here denoted by $(\stPz_x, x \in\mathbb{R}^*)$,  can be defined via  the following Doob $h$-transform 
  \[
  \left. \frac{\dd\stPz_x}{\dd\stP_x} \right|_{\mathcal{F}_t} := \frac{ \sin(\pi\alpha\hat{\rho})\mathbf{1}_{(X_t>0)} + \sin(\pi\alpha\rho)\mathbf{1}_{(X_t<0)} }{ \sin(\pi\alpha\hat{\rho})\mathbf{1}_{(x>0)} + \sin(\pi\alpha\rho)\mathbf{1}_{(x<0)} } \left| \frac{X_t}{x} \right|^{ \alpha-1} \mathbf{1}_{(t<T_0)}, \quad t\geq 0,
  \] 
where $(\mathcal{F}_t)_{t\ge 0}$, denotes the natural filtration generated by the stable process $X$ satisfying the usual conditions. The MAP associated  to the process $(X, \stPz)$, via the Lamperti-Kiu representation has  matrix exponent $ F^0$ which is similar to $\widehat{F}$ but with the roles of $\rho$ and $\hat{\rho}$ interchanged (see  Theorem 3.1 in \cite{KRS2015}). In other words,  the MAP associated to stable process conditioned to be continuously absorbed at the origin is the dual of the MAP associated to $(-X, \mathbf{P})$.  

Considering this, it can be verified that $\det(F^0(z)) = \det(\widehat{F}(z))$, for $z\in(-\alpha, 1)$, and therefore the Cram\'er number is $\theta = 1-\alpha$. Thus  the stable process conditioned to be continuously absorbed at the origin $(X, \stPz)$ has a unique recurrent extension that leaves $0$ continuously if and only if $\alpha\in (1/2,1)$.

\subsection{Spectrally negative case}

In this example, we suppose that $(X, \stP)$ is a rssMp  with no positive jumps and we will refer to this class as \define{spectrally negative rssMp}. From the Lamperti-Kiu representation, it is clear that its associated MAP $((\xi, J), \MAPP)$ also has  no positive jumps.  Therefore the rate matrix of the  Markov chain $J$ is given by
  \[
  Q = \begin{pmatrix}
  -q_{1-1} & q_{1-1} \\
  0 & 0
  \end{pmatrix}.
  \]
For simplicity, we write $q^+ = q_{1-1}$. Furthermore, since the process $X$ has no positive jumps, then $U_{-1, 1}=0$. Putting all the pieces together, we obtain that the matrix exponent of  $((\xi, J), \MAPP)$ satisfies  
  \[
  F(z) = 
  \begin{pmatrix}
 \psi^\dag(z) & q^+G_{1-1}(z) \\  
  0 &  \psi_{-1}(z) 
    \end{pmatrix},\qquad \textrm{for }\,\, z\ge0,
  \]
where $\psi^\dag(z) = \psi_1(z) - q^+,$ is the Laplace exponent of $\xi_1$ a spectrally negative L\'evy process killed at exponential time with parameter $q^+$ which is associated,  via the Lamperti representation,  to the process $X$ killed at the first time it enters $(-\infty, 0)$. Since $\xi_1$ and $\xi_{-1}$ are spectrally negative, $\psi_{i}$, $i=-1,1$ are well defined for $z\geq 0$. If $G_{1-1}(z)$ is finite for $0\leq z < z_0$, for some $z_0>0$, then $F(z)$ is well defined for $z\in [0,z_0)$. From the form of the matrix exponent $F$, it is clear that  $\det(F(z)) = 0$ if and only if $\psi^\dag_1(z)=0$ or $\psi_{-1}(z)=0$. In other words, $((\xi, J), \MAPP)$ satisfies the Cram\'er condition if and only if some of its associated L\'evy processes  satisfies the Cram\'er condition. 

On the other hand, it is well known that $\psi_1^\dag$ is a convex function with $\psi_1(0)=0$. We denote by $ \Phi^\dag_1(0) $ for its largest zero. Similarly $\Phi_{-1}(0)$ denotes the largest zero of $\psi_{-1}$.
 Hence,  $(X, \stP)$ has a recurrent extension that leaves 0 continuously whenever $\Phi_1^\dag(0)$ or $\Phi_{-1}(0)\in (0,\alpha)$. Since $ ((\xi, J), \MAPP)$ drifts to $-\infty$, then $\xi_{-1}$ drifts to $-\infty$, so $\Phi_{-1}(0)$ always exists. Therefore, by Theorem \ref{thm: ext}, $\Phi_{-1}(0)\in (0,\alpha)$ if and only if $(X, \stP)$ has a recurrent extension that leaves 0 continuously.

To illustrate this, we consider the spectrally negative  stable process, with $\alpha\in (1,2)$. In this case $\rho=1/\alpha$. Recall that the L\'evy process $\xi_1$ is associated to the stable process killed at the first time it  enters $(-\infty, 0)$, via the Lamperti representation. The process $\xi_1$ is the a Lamperti stable process that appears in \citep{KP2013} (and is denoted as $\xi^*$) and its Laplace exponent satisfies 
  \[
  \psi_1^\dag(z) = \frac{1}{\pi} \Gamma(\alpha-z) \Gamma(1+z) \sin( \pi(z-\alpha+1)).
  \]
We also observe that the  L\'evy process $\xi_{-1}$ is associated to the negative of a spectrally positive stable process  killed at the first time it hits 0, which is also the Lamperti stable process that appears in \citep{KP2013} (but with $\rho=1-1/\alpha$) and its Laplace exponent satisfies 
  \[
  \psi_{-1}(z) = \frac{1}{\pi} \Gamma(\alpha-z) \Gamma(1+z) \sin( \pi(z-\alpha)).
  \]
From the latter two expression, we have that $\theta=\alpha-1$ since  $\psi_1^\dag(\theta) = \psi_{-1}(\theta)=0$. Moreover since $\theta\in(0,\alpha)$ we deduce that  $(X,\stP)$ has a recurrent extension that leaves 0 continuously as expected.

\section{Some properties of excursion measures for rssMp}

In this section we derive the  existence  of a recurrent extension for real-valued Markov processes  and some properties of their excursion measure that are needed for the sequel. The result established in the first part of this section  is an extension to the real-valued case of a result that appears in Rivero  \cite{Rivero2005}. For simplicity, we use the same notation as in \cite{Rivero2005}.

In what follows, we set $\RRa:=\mathbb{R}\setminus\{0\}$. Let $(Y_t, t\geq 0)$ and ($\widehat{Y}_t, t\geq 0$) be two real valued Markov processes having 0 as a cemetery point. We denote by $\mathbf{Q}$ and  $\mathbf{E}_\mathbf{Q}$ (resp. $\widehat{\mathbf{Q}}$ and  $\widehat{\mathbf{E}}_\mathbf{Q}$)  for the probability and expectation associated to $Y$  (resp. for $\widehat{Y}$). Similarly, we introduce  $T_0$ (resp. $\widehat{T}_0$) for  the first hitting time of 0 for $Y$  (resp. $\widehat{Y}$). Assume that $\mathbf{Q}_x(T_0<\infty) = \widehat{\mathbf{Q}}_x(T_0<\infty) = 1$, for all $x\in \RRa$. Let $(Q_t, t\geq 0)$, $W_\lambda$, (resp. $(\widehat{Q}_t, t\geq 0)\  \widehat{W}_\lambda$)  denote the semigroup and $\lambda$-resolvent for $Y$  killed at $T_0$, (resp. for $\widehat{Y}$). For $\lambda>0$, define the functions $\varphi_\lambda, \widehat{\varphi}_\lambda: \RRa\rightarrow [0,1]$, by
  $$\varphi_\lambda(x) := \mathbf{E}_{\mathbf{Q}_x}\Big[e^{-\lambda T_0}\Big] \qquad \textrm{and} \qquad \widehat{\varphi}_\lambda(x) := \widehat{\mathbf{E}}_{\mathbf{Q}_x}\Big[e^{-\lambda T_0}\Big], \qquad x\in\RRa.$$

\subsection{Existence theorem}

We consider the following hypotheses:
\begin{description}
\item[H.1.] The processes $Y$ and $\widehat{Y}$ satisfy the basic hypotheses in Blumenthal \cite{Blumenthal92}.

\item[H.2.] The resolvents $W_\lambda$ and $\widehat{W}_\lambda$ are in weak duality with respect to a $\sigma$-finite measure $\vartheta$ defined on $\mathbb{R}^\ast$.

\item[H.3.] The following integral conditions are satisfied,
  $$\int_{\mathbb{R}^\ast} \varphi_\lambda(x) \vartheta(\ud x)<\infty \qquad \textrm{and}\qquad \int_{\mathbb{R}^\ast} \widehat{\varphi}_\lambda(x) \vartheta(\ud x)<\infty, \quad \text{for }\lambda>0.$$ 
\end{description}

The main theorem of this section which is established below corresponds to the real-valued version  of Theorem 3 given in \cite{Rivero2005}.

\begin{thm}
\label{thm: exex}
Under  hypotheses \textbf{H.1-H.3}, there exist excursion measures $\mathbf{m}$ and $\widehat{\mathbf{m}}$ compatible with the semigroups $(Q_t, t\geq 0)$ and $(\widehat{Q}_t, t\geq 0)$, respectively, such that the Laplace transforms of the entrance laws $(\mathbf{m}_s, s>0)$ and $(\widehat{\mathbf{m}}_s, s>0)$ associated with $\mathbf{m}$ and $\widehat{\mathbf{m}}$, respectively, are determined by
  $$\int_0^\infty e^{-\lambda s} \mathbf{m}_s f \ud s = \int_{\mathbb{R}^\ast} \widehat{\varphi}_\lambda(x)f(x) \vartheta(\ud x), \qquad \int_0^\infty e^{-\lambda s} \widehat{\mathbf{m}}_s f \ud s = \int_{\mathbb{R}^\ast} \varphi_\lambda(x)f(x) \vartheta(\ud x), $$
for any  continuous, bounded function $f$ and $\lambda >0$. Furthermore, associated with these excursion measures there exist Markov processes $Y^\ast$ and $\widehat{Y}^\ast$ which are extensions of $Y$ and $\widehat{Y}$, respectively, and which are still in weak duality with respect to measure $\vartheta$.
\end{thm}

The proof of the previous theorem follows similar arguments as those given in \cite{Rivero2005}, and in particular, it is based on the following three lemmas. For these reasons we will just provide the main ideas of their proofs.

\begin{lem}
\label{lem: exex1}
The family of measures $(M_\lambda , \lambda> 0)$, defined by 
\[
M_\lambda f := \int_{\RRa} f(x) \hat{\varphi}_\lambda(x) \vartheta(\ud  x) , 
\]
is such that the following hold:
 \begin{enumerate}
 \item[(i)] $\lim_{\lambda \to \infty }M_\lambda 1 = 0$;
 
 \item[(ii)] for $\mu, \lambda>0$ such that  $\mu\neq \lambda$ and  $f$ continuous and bounded, we have
   \[
   (\mu-\lambda)M_\lambda W_\mu f = M_\lambda f - M_\mu f.
   \]
 \end{enumerate}
\end{lem}

\begin{proof} We first observe  that the hypothesis \textbf{H.3} implies that $M_\lambda 1$ is finite for all $\lambda>0$. Hence claim (i) follows from  the dominated convergence theorem. To prove (ii), we first observe that for $\lambda>0$, $ \widehat{W}_\lambda 1 = \lambda^{-1}( 1-\widehat{\varphi}_\lambda).$
Hence, using  the well-known identity  for resolvents $(\mu-\lambda)\widehat{W}_\mu\widehat{W}_\lambda = \widehat{W}_\lambda-\widehat{W}_\mu$, for $\mu\neq \lambda$, we deduce
  \[
  (\mu-\lambda) \widehat{W}_\mu \widehat{\varphi}_\lambda = \widehat{\varphi}_\lambda - \widehat{\varphi}_\mu, \qquad  \textrm{for}\quad \mu\neq \lambda. 
  \]
Thus, using the latter identity and the weak duality between  the resolvents $W_\lambda$ and $\widehat{W}_\lambda$ we get for $\mu\neq \lambda$,
\[
   (\mu-\lambda)M_\lambda W_\mu f = M_\lambda f - M_\mu f,
   \]
as expected. This completes the proof.
\end{proof}

Next, we observe that  Lemma \ref{lem: exex1} and Theorem 6.9 in \citep{GS1973} guarantee that there exists a unique entrance law $(\mathbf{m}_t, t>0)$ for the semigrup $(Q_t, t\geq 0)$, such that for $\lambda>0$ and $f$ measurable and bounded,
  \[
 \int_0^1\mathbf{m}_t1\dd t < \infty, \qquad \textrm{and} \qquad M_\lambda f = \int_0^\infty e^{-\lambda t} \mathbf{m}_t f \mathrm{d}t.
  \]

According with \cite{Blumenthal92}, for an entrance law $(\mathbf{m}_s, s>0)$, there exists a unique excursion measure $\mathbf{m}$ having this entrance law. The same arguments guarantee the existence of an excursion measure $\widehat{\mathbf{m}}$ and an entrance law $(\widehat{\mathbf{m}}_t, t>0)$ for the semigroup $(\widehat{Q}_t, t\geq 0)$.  

Using the results in \citep{Blumenthal92}, we obtain that associated with the excursion measure $\mathbf{m}$ there exists a unique Markov process $Y^\ast$  extending $Y$ and the $\lambda$-resolvent of $Y^\ast$ is determined by the following identities
 \begin{equation}\label{iidentity-lemma}
 W^\ast_\lambda f(0) = \frac{M_\lambda f}{ \lambda M_\lambda 1 }, \qquad \textrm{and} \qquad W^\ast_ \lambda f(x) = W_\lambda f(x) + \varphi_\lambda(x) W^\ast_\lambda f(0), \quad \textrm{for } x\in \RRa, 
 \end{equation}
for $f$ measurable and bounded. Similarly, we obtain the existence of  $\widehat{Y}^\ast,$ and its associated $\lambda$-resolvent  $\widehat{W}^\ast_\lambda$ is defined in a similar way.

Now, since $\mathbf{m}_s1$ is decreasing in $s$ and $\int_0^1 \mathbf{m}_s1\mathrm{d}s $ is finite, we deduce
  \begin{equation*}
  \mu M_\mu 1 = \lim_{s\to \infty} \mathbf{m}_s 1 + \int_0^\infty (1-e^{-\mu t}) \nu(\dd  t), 
  \end{equation*}
where $\nu(\dd t) = -\dd \mathbf{m}_t 1$. A similar identity holds for $\widehat{M}_\mu$ with $\widehat{\nu}(\ud t)= -\dd \widehat{\mathbf{m}}_t 1$. On the one hand, using Lemma \ref{lem: exex1} part (ii), we get  \[
  (\lambda - \mu)M_\lambda \varphi_\mu = \lambda M_\lambda 1 - \mu M_\mu 1 \qquad\textrm{and}\qquad
  (\lambda - \mu)\widehat{M}_\lambda \widehat{\varphi}_\mu = \lambda \widehat{M}_\lambda 1 - \mu \widehat{M}_\mu 1.
  \]
The latter  identities, together with $M_\lambda \varphi_\mu = \widehat{M}_\mu \widehat{\varphi}_\lambda $, imply
  \[
  \lambda M_\lambda 1 - \mu M_\mu 1 = \lambda \widehat{M}_\lambda 1 - \mu \widehat{M}_\mu 1.
  \]    
Therefore by letting $\mu\to 0$, it is clear
  \[
   \lambda M_\lambda 1 - \lim_{s\to \infty} \mathbf{m}_s1 = \lambda \widehat{M}_\lambda 1 - \lim_{s\to \infty} \widehat{\mathbf{m}}_s1,
  \]
implying
\begin{equation*}
  \int_0^\infty (1-e^{-\lambda s}) \nu(\dd  s) = \int_0^\infty (1-e^{-\lambda s}) \widehat{\nu}(\dd  s).
\end{equation*}
On the other hand,  since $\mathbf{m}$ is the excursion measure associated to the entrance law $(\mathbf{m}_s, s>0)$, we have
  \[
  \mathbf{m}(1-e^{-\lambda T_0}) = \lambda M_\lambda 1 = \lim_{s\to \infty} \mathbf{m}_s 1 + \int_0^\infty (1-e^{-\lambda t}) \nu(\dd  t).
  \]
Thus if we let  $\lambda\to 0$, the dominated convergence theorem implies  $\lim_{s\to \infty} \mathbf{m}_s1=0$. Similarly, one can deduce that $\lim_{s\to \infty} \widehat{\mathbf{m}}_s1=0$. Putting all the pieces together, the following result can be deduced.

\begin{lem}
\label{lem: exex2}
For every $\lambda>0$, we have $\lambda M_\lambda 1 = \lambda \widehat{M}_\lambda 1$.
\end{lem}

The last lemma in this section establishes weak duality between the resolvents $W^*_\lambda$ and $\widehat{W}^*_\lambda$ with respect to the $\sigma$-finite measure $\vartheta$.

\begin{lem}
For every $\lambda>0$ and every measurable functions $f,g:\RR \to \RR$, we have
  \[
  \int_{\RRa} g(y)W^*_\lambda f(y) \vartheta(\dd  y) = \int_{\RRa} f(y) \widehat{W}^*_\lambda g(y) \vartheta(\dd  y).
  \]
\end{lem}
\begin{proof}
Using  the second identity in (\ref{iidentity-lemma}), we obtain
  \[
  \int_{\RRa} g(y)W^*_\lambda f(y) \vartheta(\dd  y) = \int_{\RRa} g(y) W_\lambda f(y) \vartheta(\dd  y) + W^*_\lambda f(0) \widehat{M}_\lambda g.
  \]
Now, the first identity in (\ref{iidentity-lemma})  and Lemma \ref{lem: exex2} imply $W^*_\lambda f(0) \widehat{M}_\lambda g = \widehat{W}_\lambda g(0) M_\lambda f$. Thus, using the weak duality between $W_\lambda$ and $\widehat{W}_\lambda$, we conclude
\begin{eqnarray*}
\int_{\RRa} g(y)W^*_\lambda f(y) \vartheta(\dd  y) &=& 
\int_{\RRa} f(y) \widehat{W}_\lambda g(y) \vartheta(\dd  y) + \widehat{W}^*_\lambda g(0) M_\lambda f \\
&=&
\int_{\RRa} f(y)\widehat{W}^*_\lambda g(y) \vartheta(\dd  y).
\end{eqnarray*}
This completes the proof.
\end{proof}

\subsection{Self-similarity property}

For the development of this section, we introduce the following transformation: for $c\in \RR$, let $H_c$ be such that $H_c f(x) = f(cx)$. Our first lemma provides equivalences of the self-similarity property of the excursion measure associated to the recurrent extension. Its proof follows the same arguments as those used in  Lemma 2 in \cite{Rivero2005} for positive self-similar Markov processes. So, we omit its proof.

\begin{lem}
\label{lem: excprop}
Let $\mathbf{n}$ be an excursion measure and $\overline{X}$ the associated recurrent extension of the minimal process. The following are equivalent:
\begin{enumerate}
\item[(i)] The process $\overline{X}$ satisfies the scaling property.

\item[(ii)] There exists $\gamma\in(0,1)$ such that, for any $c>0$ and  $f\in C_b(\RRa)$,
  \[
  \mathbf{n}\left( \int_0^{T_0} e^{-q s} f(X_s) \dd  s \right) = c^{\alpha(1-\gamma)} \mathbf{n}\left( \int_0^{T_0} e^{-q c^\alpha s} H_c f(X_s) \dd  s \right).
  \]

\item[(iii)] There exists $\gamma \in (0,1)$ such that, for any $c>0$ and $f\in C_b(\RRa)$,
  \[
  \mathbf{n}_sf = c^{-\alpha \gamma} \mathbf{n}_{s/c^{\alpha}} H_cf, \qquad \text{for all }\quad s>0.
  \]
\end{enumerate}
\end{lem}

\begin{lem}
\label{lem: excprop2}
Let $\mathbf{n}$ be a normalized excursion measure and $\overline{X}$ the associated extension of the minimal process. Assume that one of the conditions in Lemma \ref{lem: excprop} holds. Then there exist finite constants $C_{\alpha, \gamma}^1, C_{\alpha, \gamma}^{-1}$ different from  zero such that
\begin{equation}
\label{sojmeas}
 \mathbf{n}\left( \int_0^{T_0} 1_{\{X_s\in \dd y\}} \dd s \right) = C_{\alpha, \gamma}^{[y]} \abs y^{\alpha-1-\gamma\alpha} \dd y, \quad y\in \RRa, 
\end{equation}
with $\gamma$ determined by (ii) of Lemma \ref{lem: excprop}. Furthermore, $C_{\alpha, \gamma}^1, C_{\alpha, \gamma}^{-1}$ satisfy
  \begin{equation}\label{identityexp2}
   C_{\alpha, \gamma}^1 \MAPE_{0,1}\Big[I^{-(1-\gamma)}\Big]+ C_{\alpha, \gamma}^{-1} \MAPE_{0,-1}\Big[ I^{-(1-\gamma)}\Big] = \frac{\alpha}{\Gamma(1-\gamma)},
  \end{equation}
where
  $$I = \int_0^\infty \exp\{\alpha \xi(t)\} \dd t.$$ 
As a consequence, $\MAPE_{0,1}\Big[I^{-(1-\gamma)}\Big],  \MAPE_{0,-1}\Big[ I^{-(1-\gamma)}\Big] < \infty$.
\end{lem}

\begin{proof}
Recall that the sojourn measure
  \[
  \mathbf{n}\left( \int_0^{T_0} 1_{\{X_s\in \dd y\}} \dd s \right) = \int_0^\infty \mathbf{n}_s(\dd y) \dd s
  \]
is a $\sigma$-finite measure on $\RRa$ and is the unique excessive measure for the semigroup of the process $\overline{X}$. Now, using  part (iii) from Lemma \ref{lem: excprop} and Fubini's theorem, we obtain for $f\geq 0$ measurable:
\begin{eqnarray*}
\int_0^\infty \mathbf{n}_sf \dd s &=& \int_0^\infty s^{-\gamma} \mathbf{n}_1 (H_{s^{1/\alpha}} f) \dd s \\
&=&
\int_{\RRa} \mathbf{n}_1(\dd z)\int_0^\infty  \dd s \, s^{-\gamma} f(s^{1/\alpha} z)  \\
&=&
C_{\alpha, \gamma}^1 \int_0^\infty   f(u) u^{\alpha-1-\gamma\alpha} \dd u + C_{\alpha, \gamma}^{-1} \int_{-\infty}^0 f(u) (-u)^{\alpha-1-\gamma\alpha} \dd u,
\end{eqnarray*}
where in the last identity we have performed the change of variables $u=s^{1/\alpha} z$,
\[
C_{\alpha, \gamma}^1:=\alpha \int_0^\infty z^{-\alpha(1-\gamma)}\mathbf{n}_1 (\dd z)\qquad \textrm{and}\qquad C_{\alpha, \gamma}^{-1}:=\alpha \int_{-\infty}^0 (-z)^{-\alpha(1-\gamma)}\mathbf{n}_1 (\dd z).
\]
Thus, it holds the following representation of the sojourn measure, 
  \[
  \mathbf{n}\left( \int_0^{T_0} f(X_s) \dd s \right) = \int_0^\infty \mathbf{n}_sf \dd s = \int_{ \RRa}  f(u) C_{\alpha, \gamma}^{[u]} \abs u^{\alpha-1-\alpha\gamma} \dd u, 
  \]
which implies (\ref{sojmeas}).  

Next, we observe that the function $\varphi(x) = \mathbf{E}_x[e^{-T_0}]$ is integrable with respect to the sojourn measure since from the Markov property of $\mathbf{n}$, we have
\[
\begin{split}
\mathbf{n}\left( \int_0^{T_0} \varphi(X_s) \dd s\right) & = 
\int_0^\infty \mathbf{n}(e^{-(T_0-s)}, s<T_0) \dd s= \mathbf{n}(1-e^{-T_0}) = 1. 
\end{split}
\]  
On the other hand, using the representation of the sojourn measure, Fubini's theorem and  the fact that  $T_0$ under $\stP_y$ has the same law as $\abs y ^{\alpha} I$ under $\MAPP_{0, [y]}$, we deduce
\[
\begin{split}
1&=\mathbf{n}\left( \int_0^{T_0} \varphi(X_s) \dd s\right) = \int_{  \RRa } \mathbf{E}_y[e^{-T_0}] C_{\alpha, \gamma}^{[y]} \abs y^{\alpha - 1 - \alpha\gamma} \dd y \\
&= C_{\alpha, \gamma}^1  \MAPE_{0,1}\left[ \int_0^\infty e^{-y^\alpha I} y^{\alpha - 1 - \alpha\gamma} \dd y \right] +  C_{\alpha, \gamma}^{-1}  \MAPE_{0,-1} \left[ \int_{-\infty}^0 e^{-(-y)^\alpha I} (-y)^{\alpha - 1 - \alpha\gamma} \dd y \right] \\
&= \frac{\Gamma(1-\gamma)}{\alpha} \Big[ C_{\alpha, \gamma}^1 \MAPE_{0,1}[I^{-(1-\gamma)}] + C_{\alpha, \gamma}^{-1}\MAPE_{0,-1} [I^{-(1-\gamma)}] \Big].
\end{split}
\] 
Therefore, $\MAPE_{0,i}[I^{-(1-\gamma)}]<\infty$, $i=-1,1$ and identity (\ref{identityexp2}) holds. This ends the proof.  
\end{proof}

Voulle-Apiala \cite{Vuolle1994} and Rivero \cite{Rivero2007}  proved that any pssMp for which 0 is a regular and recurrent state either leaves 0 continuously or by jumps. The same occurs in the real-valued case. This is stated in the following lemma and its proof is similar to the one provided in \cite{Rivero2007}. For the sake of completeness, we provide its proof.

\begin{lem}
\label{lem: excprop3}
Let $\mathbf{n}$ be a self-similar excursion measure compatible with $(X, \stP)$ and with  self-similarity index $\gamma \in (0,1)$. Then,
  \[
  \text{either }\quad \mathbf{n}(X_{0+}\neq 0) = 0 \quad \text{ or } \quad  \mathbf{n}(X_{0+}=0)=0.
  \]
\end{lem}

\begin{proof}
We proceed by contradiction. Suppose that our claim is not true. Let 
\[
\mathbf{n}^c := c^{(c)}\mathbf{n}_{|_{\{X_{0+}=0\}}}\qquad \textrm{and}\qquad \mathbf{n}^j := c^{(j)} \mathbf{n}_{|_{\{X_{0+}\neq 0\}}},
\]
be the restriction of $\mathbf{n}$ to the set of trajectories $\{X_{0+}=0\}$ and $\{X_{0+}\neq 0\}$, respectively, and $c^{(c)}$, $c^{(j)}$ be normalizing constants such that
  \[
  \mathbf{n}^c(1-e^{-T_0}) = \mathbf{n}^j(1-e^{-T_0}) = 1.
  \]
The measures $\mathbf{ n}^c$ and $\mathbf{ n}^j$ are self-similar excursion measures compatible with $(X, \stP)$ and with the same self-similarity index $\gamma$. According to Lemma \ref{lem: excprop2}, the potential measure $\mathbf{n}^c$ and that of $\mathbf{n}^j$ are given by the same purely excessive measure. In other words,
  \[
  \mathbf{n}^c\left( \int_0^{T_0} \mathbf{1}_{\{X_s \in \dd y \}} \dd s \right) = C_{\alpha, \gamma}^{[y]} \abs{y}^{\alpha-1-\gamma\alpha} \dd y =
  \mathbf{n}^j\left( \int_0^{T_0} \mathbf{1}_{\{X_s \in \dd y\}} \dd s \right),
  \]
where $C_{\alpha, \gamma}^1, C_{\alpha, \gamma}^{-1}$ are constants satisfying $0<C_{\alpha, \gamma}^1 + C_{\alpha, \gamma}^{-1}<\infty$. So, by Theorem 5.25 in Getoor \cite{Getoor1990} on the uniqueness of purely excessive measures, the entrance laws associated with $\mathbf{n}^c$ and $\mathbf{n}^j$ are equal. Hence, by Theorem 4.7 of Chapter V in \cite{Blumenthal92}, the measures $\mathbf{n}^c$ and $\mathbf{n}^j$ are equal. This leads to a contradiction since the supports of the measures $\mathbf{n}^c$ and $\mathbf{n}^j$ are disjoint. 
\end{proof}

The last result of this section characterizes the form of the jumping measures associated to a self-similar excursion measure of recurrent extension of rssMp. Its proof  follows similar arguments as those used  by Voulle-Apiala in \cite{Vuolle1994} where this result is established for pssMp.

\begin{lem}
\label{lem: excprop4}
The only possible jumping-in measures such that the associated excursion measure satisfies (ii) in Lemma \ref{lem: excprop} are of the type
  \[ \eta(\dd x) = b_{\alpha, \beta}^1 x^{-(1+\beta)}\dd x + b_{\alpha, \beta}^{-1} (-x)^{-(\beta+1)} \dd x, \quad \textrm{for} \quad x\neq 0, 0<\beta < \alpha.
  \]
\end{lem}
\begin{proof}
According to Blumenthal \cite{Blumenthal92}, there exists a $\sigma$-finite measure $\eta$ on $\RRa$, such that, the entrance law $(\mathbf{n}_s, s>0)$ has the following representation
  \[
 \mathbf{n}_s = \theta_s + \eta P_s^0,
  \]
where $\theta_s$ is the entrance law for $(P_t, t\geq 0)$ with the additional property 
  $$\theta_s(U^c) \to 0, \quad \text{as }s\to 0, $$
for every neighbourhood $U$ of $0$. Furthermore, $\eta$ satisfies
  \[
  \lim_{s\to 0} \mathbf{n}_sg = \eta g, 
  \]
for all function $g$ with compact  support on $\RRa$.  From Lemma \ref{lem: excprop} (iii) and the previous identity, we deduce that for any $c>0$:
\[
\begin{split}
\lim_{s\to 0}  \int_{\RRa} g(x) \eta(-\dd x) =  \lim_{s\to 0}  \int_{\RRa} g(x) c^{\alpha \gamma} \eta(-c\dd x) .
\end{split}
\]
In particular, this shows that for every $c>0$, 
\begin{equation}
\label{etaprop-}
\eta(-\dd x) = c^{\alpha \gamma} \eta( -c\dd x ), \quad x\in \RRa. 
\end{equation}
Let $\beta=\alpha\gamma$ and define on $\RR^+$ the measure
  \[
  \nu(A) := \int_{A} x^\beta \eta(-\dd x), \quad A\in \mathbf{B}(\RR^+).
  \]
By (\ref{etaprop-}) we have that $\nu$ satisfies
\[
\begin{split}
\nu(yA) = \int_{yA} x^\beta \eta(-\dd x) =
\int_{A} (yx)^\beta \eta(-y\dd x) =
\int_{A} x^\beta \eta(-\dd x) = \nu(A).
\end{split}
\]
That is to say, the measure $\nu$ is left invariant on the group of positive real numbers under multiplication.  The uniqueness of left Haar measures  implies that there exists a constant $b_{\alpha, \beta}^{-1}$ such that
  $$\nu(\dd x) = b_{\alpha, \beta}^{-1} \frac{1}{x} \dd x, \quad  x> 0.$$
  We refer to chapter 9 in Cohn \cite{Cohn} (see Theorem 9.2.3 and Exercise 3) for these details on Haar measures.
The latter identity implies 
  $$\eta(\dd x) = b_{\alpha, \beta}^{-1}(-x)^{-(1+\beta)} \dd x, \quad x<0.$$
A similar procedure allow us to obtain that for $x>0$, $\eta(\dd x) = b_{\alpha, \beta}^1x^{-(1+\beta)} \dd x.$
This completes the proof.
\end{proof}

\section{Proofs}\label{proof}

\subsection{Existence of recurrent extensions}

The time reversal property (\ref{timereversal}), implies the  following duality result between the resolvents of the rssMp $X$ (associated to $\xi$) and its dual. We refer to Theorem 2  in Alili et al. \cite{ACGZ2016} for its proof and whose arguments follows similar  ideas to those developed in  Lemma 2 of  Bertoin and Yor \cite{BY2002} for the  positive case (i.e. pssMp).

Recall that for $q\ge 0$ and a measurable $f:\mathbb{R}\to \mathbb{R}$,  the resolvent operators are given by
\[
V^qf(x)=\mathbf{E}_x \left[ \int_0^\zeta e^{-qt} f(X_t) \dd t \right], \qquad \widehat{V}^qf(x)=\widehat{\mathbf{E}}_x \left[ \int_0^\zeta e^{-qt} f(X_t) \dd t \right]. 
\]

\begin{lem}\label{prop: wd} 
For every $q\ge 0$ and every measurable functions $f, g:\mathbb{R}\to \mathbb{R}$, we have
\begin{equation} \label{dualid}
\int_{-\infty}^{\infty}  f(x) V^qg(x) \mu(\mathrm{ d}x)=\int_{-\infty}^{\infty}  g(x) \widehat{V}^qf(x) \mu(\mathrm{ d}x),
\end{equation}
where
\[
\mu(\mathrm{ d} x):=|x|^{\alpha-1}\pi_{[x]}\mathrm{ d}x.
\]
\end{lem}

In order to establish weak duality for rssMp  an invariant function for the semigroup of the process killed at its first hitting time of zero is needed. The invariant function is given below and was determined by  Kyprianou et al.  \cite{KRS2015} (see Theorem 2.1). Its proof relies on the Lamperti-Kiu representation for rssMp and the optional stopping theorem.

\begin{lem} \label{lem: invfun}
Assume that Assumption \ref{Cramer} holds and denote by $\theta$ and $v$  for the Cram\'er number and its associated vector. Define the function $h:\mathbb{R} \to [0, \infty)$ by
  $$h(x) = | x |^\theta v_{[x]}(\theta).$$
 Then $h$ is an invariant function for the semigroup of the rssMp killed at its  first hitting time of the point zero, here denoted by $(P_t^0)_{t\ge 0}$.
\end{lem}

The proofs of Theorem \ref{thm: ext} and Theorem \ref{thm: jumpext} relies on the following  three technical  Lemmas. The first one is a linear algebra result  whose proof is included for  sake of completeness, the second provides a necessary and sufficient condition for the finiteness of $\MAPE_{0,i}[I^{s}]$, while the third establishes the conditions which are required by  Theorem \ref{thm: exex}.

Recall that $\MAPPs:=\MAPPc{\theta}$, where $\MAPPc{\theta}$ is  defined by the exponential change of measure introduced in (\ref{expchange}) with $\theta$ being  the Cram\'er number.

\begin{lem}
\label{lem: eigenposmat}
Let $A$ be a $2\times 2$ matrix with real eigenvalues $\lambda_1\leq \lambda_2$. If $\mathrm{tr}(A)\leq 2$ and $\det(I-A)\geq 0$, then $\lambda_2\leq 1$ and $\lambda_2 <1$ holds whenever $\det(I-A)>0$.
\end{lem}
\begin{proof}
It is easy to see that $\det(I-A)\geq 0$ if and only if $|2-\mathrm{tr}(A) | \geq \sqrt{ \mathrm{tr}^2(A) - 4\det(A)}$, and since $2-\mathrm{tr}(A)\geq 0$ by hyphotesis, the latter inequality implies 
  $$\lambda_2 = \dfrac{ \mathrm{tr}(A) + \sqrt{ \mathrm{tr}^2(A) - 4\det(A)} }{2}\leq 1.$$ 
The second part follows using the same arguments.
\end{proof}
Recall that 
  $$I = \int_0^\infty \exp\{\alpha \xi(t)\} \mathrm{d} t.$$

\begin{lem}
\label{lem: funcexpmom}
Let $ ((\xi, J), \MAPP)$ be a MAP and $s\in(0, 1)$. Then, $\kappa(\alpha s)< 0$ if and only if $\MAPE_{0,i}[I^{s}]< \infty$, for $i=-1,1$. 
\end{lem}
\begin{proof}
 We first consider the MAP $((\alpha \xi, J), \MAPP)$ and observe that the same  arguments used in the proof of Proposition 3.6 in \cite{KKPW2014} provides the direct implication. For the reciprocal, we suppose that $\MAPE_{0,i}[I^s]<\infty$, $i=-1,1$. Thus, for $i=-1,1$, we have by the Markov property, 
\begin{eqnarray*}
\MAPE_{0,i}[I^{s}] &>& \MAPE_{0,i}\left[ \left( \int_1^\infty \exp\{\alpha \xi(u) \} \dd u \right)^{s} \right] \\
&=&
\MAPE_{0,i}\left[ \exp\{\alpha s\xi(1)\} \left( \int_0^\infty \exp\{\alpha(\xi(1+u) 
- \xi(1)) \} \dd u \right)^{s} \right] \\
&=&
\MAPE_{0,i}\Big[e^{\alpha s\xi(1)}; J(1) = i\Big] \MAPE_{0,i}[I^{s}] + \MAPE_{0,i}\Big[e^{\alpha s\xi(1)}; J(1) = j\Big] \MAPE_{0,j}[I^{s}], \quad j\neq i.
\end{eqnarray*}
From the latter, and  since all quantities are positive, it follows 
\begin{align*}
&\MAPE_{0,1}[I^{s}] (1-(e^{F(\alpha s)})_{11}) > (e^{F(\alpha s)})_{1-1} \MAPE_{0,-1}[I^{s}] > 0, \\
&\MAPE_{0,-1}[I^{s}] (1-(e^{F(\alpha s)})_{-1-1}) > (e^{F(\alpha s)})_{-11} \MAPE_{0,1}[I^{s}] >0. 
\end{align*}
The previous inequalities imply
\begin{align*} 
 &\hspace{2 cm} (e^{F(\alpha s) })_{11} < 1,  \quad (e^{F(\alpha s) })_{-1-1} <1, \\
 &\left( 1-(e^{F(\alpha s)})_{11} \right) \left( 1-(e^{F(\alpha s)})_{-1-1} \right) > (e^{F(\alpha s)})_{1-1} (e^{F(\alpha s)})_{-11}.
\end{align*} 
Putting all the pieces together, we deduce
\[
\mathrm{tr}(e^{F(\alpha s) }) = (e^{F(\alpha s) })_{11}  + (e^{F(\alpha s) })_{-1 -1} < 2,
\]     
and 
\[
\det\left( I - e^{F(\alpha s) } \right) = \left( 1-(e^{F(\alpha s) })_{11} \right) \left( 1-(e^{F(\alpha s) })_{-1 -1}  \right) - (e^{F(\alpha s) })_{1 -1} (e^{F(\alpha s) })_{-1 1} > 0.
\]
Using  Lemma \ref{lem: eigenposmat} we get the leading eigenvalue is less than 1. In other words, $e^{\kappa(\alpha s)}<1$  and implicitly $\kappa(\alpha s)<0$, completing  the proof.
\end{proof}

\begin{lem}
\label{expecI}
 Assume that Assumption \ref{Cramer} holds and denote by $\theta$ and $v$  for the Cram\'er number and its associated vector. Suppose that $0<\theta<\alpha$, then $\MAPE_{0,i}[I^{\theta/\alpha - 1}]$ and $\MAPEhs_{0,i}[I^{\theta/\alpha - 1 }]$ are finite, for $i=-1,1$.
\end{lem}

\begin{proof}
Since the  MAP $((\alpha \xi, J), \MAPP)$ has Cram\'er number $\theta/\alpha$, a direct application of Proposition 3.6 in \cite{KKPW2014} provides  that $\MAPE_{0,i}[I^{\theta/\alpha - 1}]$ is finite, for $i=-1,1$.

For the second expectation, we observe that the MAP $((\alpha \xi, J), \MAPPhs)$ satisfies $\kappa^\sharp(\theta/\alpha) = \kappa(0)$. If $\kappa(0)=0$, then $\theta/\alpha$ is the Cram\'er number of $((\alpha \xi, J), \MAPPhs)$ and by the first part of the proof  we have $\MAPEhs_{0,i}[I^{\theta/\alpha - 1 }]$ is finite, for $i=-1,1$. Now, if $\widehat{\kappa}^\sharp(\theta/\alpha)=\kappa (0)<0$, Lemma \ref{lem: funcexpmom} guarantees that $\MAPEhs_{0,i}[I^{\theta/\alpha}]$ is finite, for $i=-1,1$, which implies that for any $t>0$,
\begin{equation}\label{bounded}
\MAPEhs_{0,i}\left[ \left(\int_0^t e^{\alphaÊ\xi (u)} \dd u \right)^{\theta/\alpha}\right]<\infty \qquad \textrm{for}\quad i=-1,1.
\end{equation}
From the proof of Proposition 3.6 in \cite{KKPW2014}, we know that the following identity holds for all $s>0$ and $t\geq 0$,
\[
\left( \int_0^\infty e^{\alpha \xi(u)} \dd u \right)^{s} - \left( \int_t^\infty e^{\alpha \xi(u)} \dd u \right)^{s} = s \int_0^t e^{s \xi(u)} \left( \int_0^\infty e^{\alpha \xi(u+v) - \alpha \xi (v) }   \dd v \right)^{s - 1} \dd u.
\]
Hence taking expectations from both sides of the above identity, with $s=\theta/\alpha,$ and applying the Markov property, we obtain
\begin{align*}
\MAPEhs_{0,i} & \left[ \left( \int_0^\infty e^{\alpha \xi(u)} \dd u \right)^{\theta/\alpha} - \left( \int_t^\infty e^{\alpha \xi(u)} \dd u \right)^{\theta/\alpha} \right] \\
&= \dfrac{\theta}{\alpha} \int_0^t \sum_{j\in \{\pm 1\}} \MAPEhs_{0,i} \left[  e^{\theta \xi(u)}; J(u) = j\right] \MAPEhs_{0,j} \left[ I^{\theta/\alpha-1} \right] \dd u.
\end{align*}
Since $| |x|^s - |y|^s | \leq |x-y|^s$, for any $x,y\in \mathbb{R}$, and $0<s<1$, the left-hand side of the above equation is bounded by \eqref{bounded}. Since $\MAPEhs_{0,i} \left[  e^{\theta \xi(u)}; J(u) = i\right] \neq 0$, it follows that $\MAPEhs_{0,i}[I^{\theta/\alpha - 1 }]$ is finite, for $i=-1,1$, and the proof is now completed.

\end{proof}

\begin{proof}[Proof of Theorem \ref{thm: ext}] 
The proof relies on Theorem \ref{thm: exex}. Let $((\xi, J), \MAPPhs)$ the dual process of $((\xi, J), \MAPPs)$. By Proposition 4 in \cite{DK2015}, we have that $((\xi, J), \MAPPs)$ drifts to $\infty$ which imply that $((\xi, J), \MAPPhs)$ drifts to $-\infty$, and therefore $I<\infty$, $\MAPPhs$-\as

Let $\stPhs$ be the law of the $\alpha$-rssMp associated with $((\xi, J), \MAPPhs)$ via the Lamperti-Kiu transform. The process $(X, \stPhs)$ hits 0 continuously and in a finite time, $\stPhs$-\as Now, by Lemma \ref{prop: wd}, $(X, \stPs)$ and $(X, \stPhs)$ are in weak duality with respect to the measure $\mu(\dd  x)=|x|^{\alpha-1}\pi_{[x]}\dd x$, for $x\neq 0$. Since the law $\stPs$ is constructed via a Doob $h$-transform of the law $\stP$, with $h(x)= |x|^\theta v_{[x]}$  (the invariant function for the semigroup of $(X,\stP)$), it follows that $(X, \stP)$ and $(X, \stPhs)$ are in weak duality with respect to the measure
  $$\nu(dx) = \alpha \abs x^{\alpha-1-\theta} v_{[x]} \pi_{[x]} \dd x.$$

From Lemma \ref{expecI}, we know that $\MAPE_{0,i}[I^{\theta/\alpha - 1}]$ and $\MAPEhs_{0,i}[I^{\theta/\alpha - 1 }]$ are finite, for $i=-1,1$. Thus, for all $\lambda>0$, we necessarily  have
  $$\int_{\RRa} \stE_x[e^{-\lambda T_0}] \nu(\dd x) < \infty, \quad\textrm{and}\quad \int_{\RRa} \stEh_x^\sharp[e^{-\lambda T_0}] \nu(\dd x) < \infty. $$
Indeed, from the Lamperti-Kiu transform and Fubini's Theorem, we see
\begin{equation*}
\begin{split}
\int_{\RRa} \stE_x[e^{-\lambda T_0}] \nu( \dd  x) &=
\int_{\RRa} \alpha \abs x^{\alpha-1-\theta} v_{[x]} \pi_{[x]} \MAPE_{0,[x]}[e^{-\lambda \abs x^{\alpha } I}] \dd  x \\
&=
 \Gamma(1-\theta/\alpha)\lambda^{\theta/\alpha-1}\Big(v_{1} \pi_{1}  \MAPE_{0,1}[I^{\theta/\alpha-1}] + v_{-1} \pi_{-1}   \MAPE_{0,-1}[I^{\theta/\alpha-1}]\Big), 
\end{split}
\end{equation*}
which is finite. In a similar way, we can deduce  
  $$\int_{\RRa} \stEh_x^\sharp[e^{-\lambda T_0}] \nu(dx) < \infty.$$
Thus, the conditions of Theorem \ref{thm: exex} hold and we guarantee that there exists a unique recurrent extension of $(X, \stP)$ such that the $\lambda$-resolvent of its excursion measure $\mathbf{ n}$ satisfies
  \[
  \mathbf{ n}\left( \int_0^{T_0} e^{-\lambda t} f(X_t) \dd  t \right) = \frac{1}{\hat{C}_{\alpha, \theta}} \int_{-\infty}^\infty f(x) \stEhs_x[e^{-\lambda T_0}] \abs x^{\alpha-1-\theta} v_{[x]} \pi_{[x]} \dd  x, 
  \]
where 
  \[
  \hat{C}_{\alpha, \theta} = \frac{\Gamma(1-\theta/\alpha)}{\alpha} \Big(v_{1} \pi_{1}  \MAPEhs_{0,1}[I^{\theta/\alpha-1}] + v_{-1} \pi_{-1}   \MAPEhs_{0,-1}[I^{\theta/\alpha-1}]\Big).
  \]
Furthermore, $\mathbf{ n}(1-e^{-T_0}) = 1$. The characterization of the entrance law is obtained from the following series of identities
\begin{eqnarray*}
\mathbf{ n}\left( \int_0^{T_0} e^{-\lambda t} f(X_t) \dd t \right) 
&=&
\frac{1}{\hat{C}_{\alpha, \theta}} \left( \int_0^\infty f(x) \MAPEhs_{0,1}[e^{-\lambda x^\alpha I}] x^{\alpha - 1 - \theta} v_{1} \pi_{1}  \dd x \right. \\
& &
+ \left. \int_{-\infty}^0 f(x) \MAPEhs_{0,-1}[e^{-\lambda (-x)^\alpha I}] (-x)^{\alpha - 1 - \theta} v_{-1} \pi_{-1}  \dd x \right) \\
&=&
\frac{1}{\alpha \hat{C}_{\alpha, \theta}} \int_0^\infty e^{-\lambda t} t^{-\theta/\alpha} \left( v_{1} \pi_{1}  \MAPEhs_{0,1}\left[f\left( \frac{t^{1/\alpha}}{I^{1/\alpha}} \right) I^{\theta/\alpha - 1} \right] \right. \\
& &
\hspace{2.5cm}+ \left. v_{-1} \pi_{-1}  \MAPEhs_{0,-1}\left[ f\left( -\frac{t^{1/\alpha}}{I^{1/\alpha}} \right) I^{\theta/\alpha - 1} \right] \right) \dd t,
\end{eqnarray*}
where in the first identity we  used the Lamperti-Kiu transform and in the second identity, we  used Fubini's theorem and performed a change of variables.

To prove the converse, first we will verify that there exist a $\theta$ in $(0, \alpha)$ such that $\kappa(\theta)\leq 0$. By Lemma \ref{lem: excprop2}, we can deduce that there exist a $\theta \in(0,\alpha)$ such that potential of the measure $\mathbf{n}$ is given by
  \[\nu(\dd y) := \mathbf{n}\left( \int_0^{T_0} \mathbf{1}_{\{X_t\in \dd y\}} \right) = C_{\alpha, \alpha\theta}^{[y]} \abs{y}^{\alpha-1-\theta} \dd y, \quad  y\in \RRa,
  \]
where $C_{\alpha, \alpha\theta}^{1}, C_{\alpha, \alpha\theta}^{-1}$ are constants such that $0<C_{\alpha, \alpha\theta}^{1} + C_{\alpha, \alpha\theta}^{-1} < \infty$. Furthermore, $\nu$ is the unique invariant measure for $\overline{X}$  (the uniqueness holds up to a multiplicative constant). Hence, it follows that $\nu$ is an excessive measure for $(X, \stP)$. On the other hand, the Revuz measure of the additive functional $B$ defined by $B_t = \int_0^t \abs{X_s}^{-\alpha} \dd s$, for  $0\leq t < T_0$, relative to $\nu$, is given by
  \[
  \nu_B(\dd y) = C_{\alpha, \alpha\theta}^{[y]} \abs{y}^{-1-\theta} \dd y, \quad  y\in \RRa.
  \] 
Indeed, the Revuz measure of the additive functional $B$, relative to $\nu$, is such that
\[
\begin{split}
\nu_B(f) &= \lim_{t\to 0} \frac{1}{t} \stE_{\nu}\left[ \int_0^tf(X_s) \dd B_s, t<T_0 \right]=
\lim_{t\to 0} \int_{\RRa} \frac{1}{t} \stE_x  \left[ \int_0^tf(X_s) \abs{X_s}^{-\alpha} \dd s, t<T_0 \right] \nu(\dd x) \\
&=
\int_{\RRa} f(x)\abs{x}^{-\alpha} \nu(\dd x) =
\int_{\RRa} f(x) C_{\alpha,\alpha\theta}^{[x]} \abs{x}^{-1-\theta}  \dd x, 
\end{split}  
\]
for all positive bounded function $f$, see e.g. \cite{Revuz70-1}.

Since the Revuz measure is excessive for the  process $(X, \stP)$, from the Lamperti-Kiu transform it follows 
\begin{equation*}
\begin{split}
\int_{\RRa} \MAPE_{0,i}\left[ f(ye^{\xi(t)} J(t)) \right] C_{\alpha, \alpha\theta}^{[y]} \abs{y}^{-1-\theta} \dd y \leq  
\int_{\RRa}  f(y) C_{\alpha, \alpha\theta}^{[y]} \abs{y}^{-1-\theta} \dd y,
\end{split}
\end{equation*}
for all positive function $f$ and all $i=-1,1$. The left hand side of the latter inequality can be written as follows
\begin{equation*} \label{aux1}
\begin{split}
\int_{\RRa} &\MAPE_{0,i}\left[ e^{\theta \xi(t)}; J(t)=1 \right] C_{\alpha, \alpha\theta}^{[x]} f(x) \abs{x}^{-1-\theta} \dd x\\
&\hspace{4cm} + \int_{\RRa} \MAPE_{0,i}\left[ e^{\theta \xi(t)}; J(t)=-1 \right] C_{\alpha, \alpha\theta}^{[-x]} f(x) \abs{x}^{-1-\theta} \dd x, 
\end{split}
\end{equation*}
for  all $i=-1,1$.  Hence, 
\[
\begin{split}
\int_{\RRa}  C_{\alpha, \alpha\theta}^{[x]} f(x) \abs{x}^{-1-\theta} \dd x &\geq  \int_{\RRa} \left[ (e^{F(\theta) t})_{i1} C_{\alpha, \alpha\theta}^{[x]}  + (e^{F(\theta) t})_{i  -1} C_{\alpha, \alpha\theta}^{[-x]} \right] f(x) \abs{x}^{-1-\theta} \dd x,
\end{split}
\]
for $i=-1,1$. Since the latter inequality holds for all positive function $f$, we deduce
\begin{align*}
 \left(1-(e^{F(\theta) t})_{i1} -  (e^{F(\theta) t})_{i -1} \right) [ C_{\alpha, \alpha\theta}^1 + C_{\alpha, \alpha\theta}^{-1}]   \geq 0,
\end{align*}
for $i=-1,1$, which implies  the series of  inequalities 
\begin{align*} 
 &(e^{F(\theta) t})_{11}\leq 1, \quad (e^{F(\theta) t})_{1 -1}\leq 1-(e^{F(\theta) t})_{11}, \\   
 &(e^{F(\theta) t})_{-1-1}\leq 1, \quad (e^{F(\theta) t})_{-1 1}\leq 1-(e^{F(\theta) t})_{-1 -1}.   
\end{align*}
Putting all pieces together, we get
\[
\mathrm{tr}(e^{F(\theta) t}) = (e^{F(\theta) t})_{11}  + (e^{F(\theta) t})_{-1 -1} \leq 2
\]     
and 
\[
\det\left( I - e^{F(\theta) t} \right) = \left( 1-(e^{F(\theta) t})_{11} \right) \left( 1-(e^{F(\theta) t})_{-1 -1}  \right) - (e^{F(\theta) t})_{1 -1} (e^{F(\theta) t})_{-1 1} \geq 0.
\]
Thus the conditions of Lemma \ref{lem: eigenposmat} hold and therefore $e^{\kappa(\theta)}\leq 1$, i.e., $\kappa(\theta)\leq 0$.

Finally,  if $\kappa(\theta)<0$ then  Theorem \ref{thm: jumpext}  (which is proved below) implies that $(X,\stP)$ admits a recurrent extension that leaves 0 by a jump with jumping-in measure proportional to $ \eta_{\theta}(\dd x) = |x|^{-(\theta+1)} \dd x$, for $ x\neq 0$. In other words,  the measure $\mathbf{ m}=2^{-1}\mathbf{n} + 2^{-1}c_{\alpha, \theta} \stP_{\eta_\theta}$ is a self-similar excursion measure compatible with $(X,\stP)$ and with index of self-similarity $\theta/\alpha$; where $c_{\alpha, \theta}$ is a normalizing constant. Therefore, there exists a recurrent extension of $(X,\stP)$ with excursion measure $\mathbf{ m}$ that may leave 0 by a jump and continuously at the same time, which leads to a contradiction since any recurrent extension of $(X,\stP)$ either leaves 0 by a jump or continuously (see Lemma \ref{lem: excprop3}). Therefore, $\kappa(\theta)=0$, that is to say $\theta \in (0,\alpha)$ is a Cram\'er number.
\end{proof}

\begin{proof}[Proof of Theorem \ref{thm: jumpext}]
The equivalence of assertions (i) and (ii) follow from Lemma \ref{lem: funcexpmom} considering the MAP $ ((\xi,J), \MAPP)$ and $s=\beta/\alpha$.

Now, we prove the equivalence of the assertions (ii) and (iii).  Using again that $T_0$ under $\stP_x$ has the same law as $|x|^\alpha I$ under $\MAPP_{0,[x]}$, we obtain
\begin{equation*}
\begin{split}
\int_{\RRa} \stE_x[1-e^{-T_0}] \abs x^{-(\beta+1)} \dd x 
&= \MAPE_{0,1}\left[ \int_0^{\infty} (1-e^{-y I}) \frac{1}{\alpha} y^{-\beta/\alpha-1} \dd y \right] \\
&\quad  + \MAPE_{0,-1}\left[ \int_0^{\infty} (1-e^{-y I}) \frac{1}{\alpha} y^{-\beta/\alpha-1} \dd y \right] \\
&= \frac{\Gamma(1-\beta/\alpha)}{\beta} \Big( \MAPE_{0,1}[I^{\beta/\alpha}] +   \MAPE_{0,-1}[I^{\beta/\alpha}] \Big).
\end{split}
\end{equation*} 
Thus, if $\eta(\dd x) = \abs x^{-(\beta+1)} \dd x$ and $\mathbf{ n}^j$ is the pseudo-excursion measure $\mathbf{ n}^j = \stP_\eta$, then
  \[
  \mathbf{ n}^j(1-e^{-T_0}) = \int_{\RRa} \stE_x[1-e^{-T_0}] \abs x^{-(\beta+1)} \dd x 
  \]
is finite if and only if $\MAPE_{0,i}[I^{\beta/\alpha}] < \infty$, for $i=-1,1$. This proves the equivalence between assertions in (ii) and (iii). 

Finally, if (iii) holds, according with \cite{Blumenthal92} and Lemma \ref{lem: excprop4}, associated with the normalized excursion $\mathbf{ n}^{j'} = \stP_\eta$ where $\eta(\dd x) = b_{\alpha, \beta}^{[x]} \abs x^{-(\beta+1)}\dd x$ there exists a unique extension of the minimal process $(X, T_0)$ which is a self-similar Markov process and which leaves 0 by a jump, according to the jumping-in measure $\eta(\dd x) = b_{\alpha, \beta}^{[x]} \abs x^{-(\beta+1)}\dd x$, implying  (iv). Conversely, if (iv) holds the It\^o excursion measure of $\tilde{X}$ is $\mathbf{ n}^{j'} = \stP_\eta$, with $\eta(\dd x) = b_{\alpha, \beta}^{[x]} \abs x^{-(\beta+1)}\dd x$ and the statement in (iii) follows. This completes the proof.
\end{proof}

\bibliographystyle{abbrv}
\bibliography{ref3}

\end{document}